\documentclass[letterpaper,10 pt,conference]{ieeeconf}  % Comment this line out
\IEEEoverridecommandlockouts
\usepackage[utf8]{inputenc}
\usepackage{dsfont,amsmath,amsfonts,amsbsy,amssymb,amscd}
\usepackage{ntheorem}
\usepackage{bm}  %bold for symbols
\usepackage[version=4]{mhchem}
\usepackage{bigints}
\usepackage{graphicx} % Necessary to use \scalebox
\usepackage{subfigure}
\usepackage{array}
\usepackage{caption}
\usepackage{epstopdf} %This will convert your eps on-the-fly
\usepackage{bigstrut}
\usepackage{comment}
\usepackage{breqn}

\newtheorem{theorem}{Theorem}
\newtheorem{lemma}{Lemma}
\newtheorem{proper}{Property}

\newtheorem{cond}{Condition}
\newtheorem{remark}{Remark}

\newcommand\T{\rule{0pt}{2.6ex}}       % Top strut
\newcommand\B{\rule[-1.2ex]{0pt}{0pt}} % Bottom strut
\def\-{\raisebox{.85pt}{-}}
\def\+{\raisebox{.85pt}{+}}
\def\={\raisebox{.85pt}{=}}

\title{\LARGE \bf
An Adaptive Observer Design for Charge-State and Crossover Estimation in Disproportionation Redox Flow Batteries undergoing Self-Discharge
}

\author{Pedro Ascencio, Kirk Smith, Charles W. Monroe$^*$, and David Howey$^*$% <-this % stops a space
\thanks{Pedro Ascencio ({\tt pedro.ascencio@eng.ox.ac.uk}), Kirk Smith ({\tt kirk.smith@eng.ox.ac.uk}), Prof. Charles W. Monroe  ({\tt charles.monroe@eng.ox.ac.uk}) and Prof.  David Howey ({\tt david.howey@eng.ox.ac.uk}) are with the Department of Engineering Science, University of Oxford, Oxford OX1 3PJ, United Kingdom. $^*$Joint corresponding authors.}%
}

\begin{document}

\maketitle
\thispagestyle{empty}
\pagestyle{empty}

%%%%%%%%%%%%%%%%%%%%%%%%%%%%%%%%%%%%%%%%%%%%%%%%%%%%%%%%%%%%%%%%%%%%%%%%%%%%%%%%
\begin{abstract}
This article considers a model formulation and an adaptive observer design for the simultaneous estimation of the state of charge and crossover flux in disproportionation redox flow batteries. This novel nonaqueous battery chemistry allows a simple isothermal lumped parameter model to be formulated. The transport of vanadium through the porous separator is a key unknown function of battery variables and it is approximated in the space of continuous functions. The state and parameter observer adaptation laws are derived using Lyapunov analysis applied to the estimation error, the stability and convergence of which are proved. Numerical values of observer gains are calculated by solving a polytopic linear matrix inequality and equality problem via convex optimization. The performance of this design is evaluated on a laboratory flow battery prototype, and it is shown that the crossover flux can be considered a linear function of state of charge for this battery configuration during self-discharge.
\end{abstract}

%%%%%%%%%%%%%%%%%%%%%%%%%%%%%%%%%%%%%%%%%%%%%%%%%%%%%%%%%%%%%%%%%%%%%%%%%%%%%%%%
\section{Introduction}
Although redox flow batteries (RFBs) have become a promising alternative for grid-scale energy storage, many fundamental challenges have to be addressed to make them competitive. A persistent issue is the crossover of active species through the separator, which leads to electrolyte imbalance and capacity fade \cite{Leung-2012,GMPS-2016,Potash-2016,Xianfeng-2011}.

To improve the design and management of RFBs, several multi-physics models have been developed, in particular for all-vanadium chemistries \cite{Weber-2011,Xu-2015,Zheng-2014}. The pioneering studies commonly assume perfect membrane selectivity. More realistic distributed-parameter models usually include, in both porous electrodes and across the separator, species, charge, and momentum conservation, mass transport by diffusion, convection and migration, ohmic losses, and interfacial reaction kinetics \cite{Shah-2010, Chen-2014,Knehr-2012,Wang-2014}. Alternatively, simpler and computationally-tractable lumped parameter approaches have been proposed, accounting for ion transport across the membrane through standard pseudo-steady transport assumptions, which establish species fluxes proportional to their concentration differences \cite{Tang-2011,Shah-2011,Kazacos-2012,Boettcher-2016}.

All of these models involve parameters that are commonly assumed known, by means of standard experiments or model fits (e.g.\ \cite{Gandomi-2018}). In particular, the separator crossover flux, typically taken to be governed by Fickian diffusion \cite{Schmal-1986,Wiedemann-1998, Sun-2010,Kamcev-2017}, is analyzed to quantify side reactions that consume the active species and consequently cause capacity fade \cite{Won-2015,Knehr-2012}. These models do not fully address the more complex transport behaviour of the membrane, for example caused by solute/solute interactions or microstructural effects \cite{Luo-2018,Shinkle-2012}. The resulting unmodeled dynamics, in addition to the assumption of invariance of the model parameters with respect to battery internal states, leads to erroneous long-term predictions of performance.

To account for the physico-chemical properties of the battery and their changes, model parameters are typically obtained either \textit{via} experimental methods, or least squares regression approaches (see discussion in \cite{Lee-2018,Vynnycky-2012})\footnote{There may be identifiability challenges in accordance with the type/structure of model selected and excitation of the battery \cite{Bizeray-2018, Park-2018}.}. Due to their simple practical implementation, real-time observer-based schemes for lumped parameter models of RFBs have also received attention \cite{Mohamed-2013,Xiong-2014, Wei-2016, Xiong-2017, Wei-2018}. These on-line identification methods are robust against uncertain initial conditions and can perform predictions of the main states of the battery based on a fixed set of parameters, or in some cases simultaneously provide their continuous estimation. In general, similar to the widespread on-line estimation methods for lithium-ion batteries \cite{Plett-2015}, these approaches use electrical equivalent circuit models (ECMs) to emulate the RFB behaviour, and commonly perform the state/parameter estimation via the extended Kalman filter (EKF) \cite{Mohamed-2013,Yu-2014,Xiong-2014,Wei-2018}, amongst other model-based observer schemes \cite{Wei-2016,Xiong-2017}. However, due to limitations in the formulation of the model, or convergence issues of the EKF, the crossover flux is not explicitly included in this type of approach.
%Instead, its consequences (mainly capacity loss, the observer-based on-line schemes of which has been scarcely investigated \cite{Wei-2018}) are deduced or detected from the observer predicted concentrations and parameter evolution.

This paper therefore addresses the challenge of simultaneous estimation of battery states and crossover flux for a novel type of nonaqueous RFB based on vanadium acetylacetonate disproportionation \cite{Liu-2009,SM-2019}. This underlying chemistry allows a particularly simple isothermal lumped-parameter model to be employed for a disproportionation redox flow battery (DRFB). The model considers the state of charge in one half-cell and its associated electrolyte reservoir, alongside the crossover flux out of the half-cell through the separator. A general mathematical structure is used to express crossover flux, avoiding any assumptions particular to a given transport mechanism through the membrane.

To perform an on-line simultaneous estimation of the battery unknown states and parameters, an adaptive observer is designed via the standard Lyapunov second method of stability. From the analysis of the stability and convergence of the observer estimation error, a coupled Linear Matrix Inequalities and Equality (LMI-LME) problem is derived, which is numerically solved via convex optimization, using a polytopic approach.

% This paper is organized as follows. In Section \ref{model} a simple isothermal lumped parameter model for a type of vanadium acetylacetonate RFB is formulated. In Section \ref{Sobs} the problem of adaptive observer design is analysed. In Section \ref{lab} a laboratory RFB prototype setup is described and in Section \ref{results} the performance of the proposed approach is illustrated. Concluding remarks and further steps of research are included in Section \ref{conclu}.

\vspace{-0.025cm}

\subsection{Notation}
The space of all continuous functions on the domain $\Omega$ into $\mathbb{R}$ is denoted $\mathcal{C}(\Omega; \mathbb{R})$. The set of (symmetric) positive definite matrices of dimension $n \times n$ is denoted $\mathbb{S}^n_+$, and $I_n$ is the identity matrix of dimension $n \times n$. Also, $\mathbb{R}^+ \equiv [0,+\infty)$, $\dot{v}(t)\triangleq \frac{dv}{dt}(t)$, $\|x\|\triangleq \sqrt{x^{\top} x}$ with $x \in \mathbb{R}^n$ being a real vector of dimension $n$, and $\|E\| \triangleq \sqrt{\lambda_{\text{max}}(E^{\top} E)}$ with $E$ being a real matrix, where $\lambda_{\text{max}}$ denotes the largest eigenvalue; $\lambda_{\text{min}}$ stands for the smallest eigenvalue.

%---------------------------------------------------------------------%
%---------------------------------------------------------------------%
\section{Modeling DRFBs}
In a DRFB, identical liquid electrolyte solutions are stored in two separate reservoirs (see Figure \ref{fig_rfb}), which contain a single metal electroactive species that can be both oxidized and reduced. When the battery is fully discharged, the active species in both reservoirs have identical oxidation states. The vanadium acetylacetonate ($\ce{V(acac)3}$) compound supports disproportionation electrochemistry.

\begin{figure}
    \centering
    \vskip+0.1cm
    \includegraphics[width=\columnwidth]{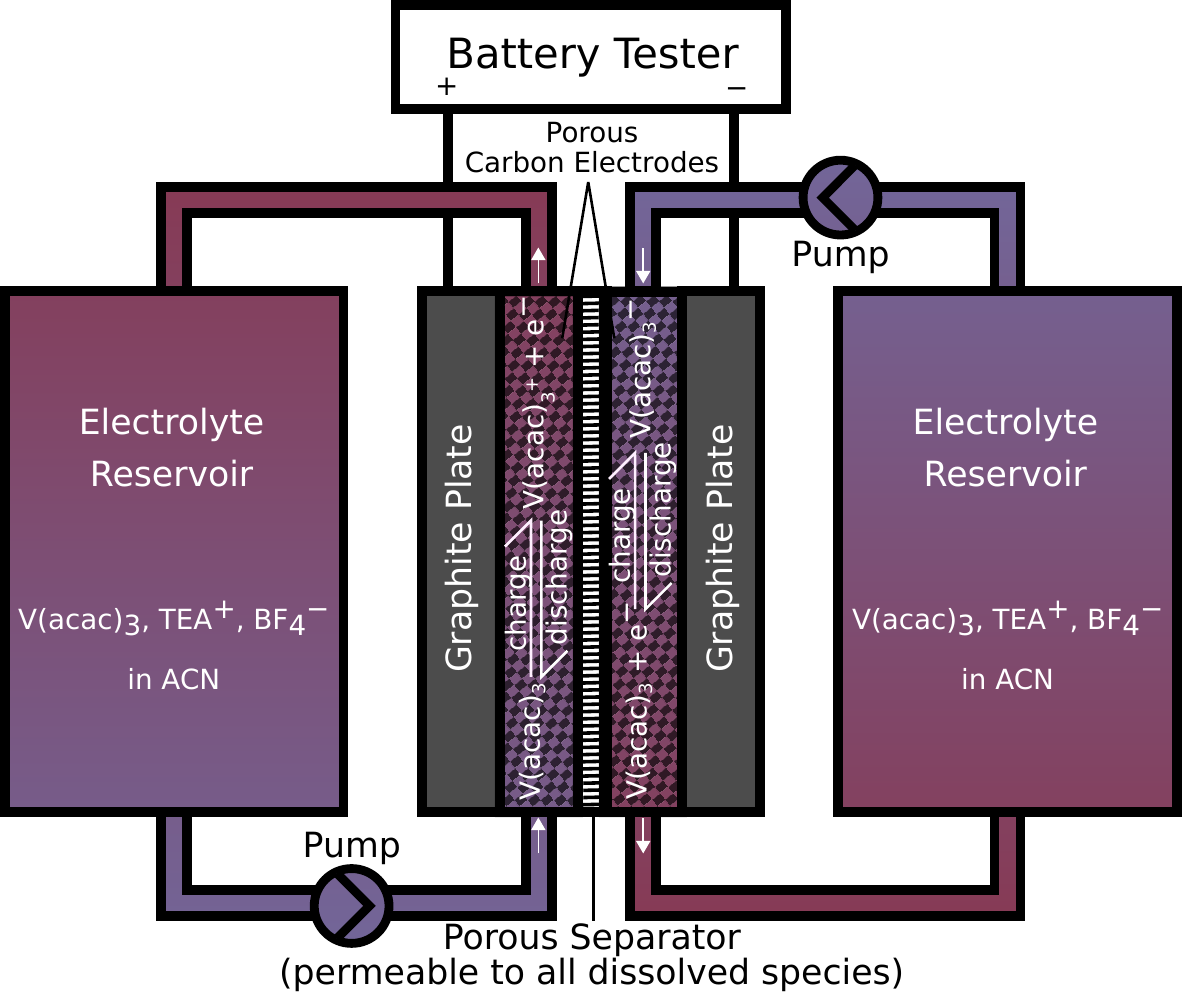}
    \vspace{-0.25cm}
    \caption{\small{Schematic of a DRFB using vanadium acetylacetonate}}
    \label{fig_rfb}
    \vskip-0.25cm
\end{figure}

A typical \ce{V(acac)3} cell differs significantly from a traditional aqueous all-vanadium RFB. To access the high redox potential associated with \ce{V(acac)3} disproportionation, one must use a nonaqueous solvent whose electrochemical stability window is wider than water. The reversible nature of the disproportionation reaction means that the battery is tolerant to crossover, which impacts coulombic efficiency but does not permanently degrade charge capacity. This in principle enables use of porous separators, rather than costly ion-exchange materials, and opens up a value tradeoff between the reactor capital cost and its coulombic efficiency \cite{Potash-2016,Liu-2009,Shinkle-2012,SM-2019,James_PhD}.

Cyclic voltammetry experiments \cite{Liu-2009} reveal two main redox couples associated with $\ce{V(acac)_3}$

\noindent
\begin{align}
\begin{split}
& \bullet \quad \text{Negative Chamber} \nonumber \\[-0.1cm]
& \ce{  V(III)(acac)_{3} + e- <=>[charge][discharge] [V(II)(acac)_{3}]^{-}}, \\[-0.1cm]
& \bullet \quad \text{Positive Chamber}  \nonumber \\[-0.1cm]
& \ce{  V(III)(acac)_{3} <=>[charge][discharge] [V(IV)(acac)_{3}]^{+} + e^{-} },
\end{split}
\end{align}

\noindent
with an equilibrium cell potential of $E^{\text{0}}=$2.18 [V] \cite{Liu-2009,Shinkle-2012}. Crossover and the resulting self-discharge of the battery is due to the comproportionation of $\ce{V(II)}$ and $\ce{V(IV)}$ to form  $\ce{V(acac)3}$ \cite{Shinkle-2012}.

\subsection{Isothermal Lumped Parameter Model}
\label{model}

Due to the symmetry inherent to disproportionation chemistry, the model analysis can be carried out by considering only one side of the battery. Let $n=n(t)$ be the amount of neutral $\ce{V(acac)_{3}}$ species remaining in one reservoir. In isothermal operation, its rate of change is due to the current\footnote{positive current is considered a discharge process} $I$ driven by the battery, and the crossover flow $Q_x$ through the separator, namely

\noindent
\begin{align}
\frac{dn}{dt}(t) &= \frac{I(t)}{\mathcal{F}} + Q_x(s(t)),
\end{align}

\noindent
where $\mathcal{F}$ is Faraday's constant and $s$ stands for some variables/states of the battery\footnote{This formulation has a particular lumped parameter validity in the discharge process, where there are no chemical and electric potential gradients and the typical resulting distributed dynamics from these phenomena are negligible \cite{Potash-2016}.}. Thus, for an initial amount of neutral $\ce{ V(acac)_{3}}$  in the prepared solution $n_0$, considering the half-cell reactor volume $V_{\text{cell}}$ to be negligible with respect to the overall volume in the reservoir $V_{\text{res}}$, the state-of-charge of the overall system can be described by:

\noindent
\begin{align}
\label{soc}
\begin{split}
SOC(t) &= \frac{n_0-n(t)}{n_0} = \frac{c_0-c(t)}{c_0},  \\
\frac{d SOC}{dt} (t) &= \!-\!\left(\frac{1}{c_0 V_{\text{res}}}\right) Q_x(s(t)) \!-\! \left(\frac{1}{c_0 V_{\text{res}}\mathcal{F}}\right) I(t),
\end{split}
\end{align}

\noindent
with $c_0\!=\!n_0/V_{\text{res}}$ and $c(t)\!=\!n(t)/V_{\text{res}}$, being the initial and overall concentration of neutral species in the battery, respectively.

\vspace{0.2cm}

\noindent
Similarly to a continuous stirred tank reactor model, we apply conservation of mass to one half-cell only, and assuming perfect mixing, this leads to

\noindent
\begin{align}
& \frac{d \left(n_{\text{cell}}-n\right)}{dt}(t) =Q(t) \left( c(t) - c_{\text{cell}}(t) \right), \nonumber \\
\label{soc_cell}
& SOC_{\text{cell}}(t) = \frac{c_0-c_{\text{cell}}(t)}{c_0},\\
& \frac{d SOC_{\text{cell}}}{dt}(t) = \!-\!\left(\frac{1}{\epsilon c_0 V_{\text{cell}}}\right) \frac{dn}{dt}(t) \!-\! \left(\frac{Q(t)}{\epsilon V_{\text{cell}}}\right) \Delta SOC(t), \nonumber
\end{align}

\noindent
where $n_{\text{cell}}$ represents the amount of neutral $\ce{V(acac)_{3}}$ in the half-cell, $SOC_{\text{cell}}$ represents the state of charge in the reactor (which may be different from the overall SOC), and $Q$ stands for the volumetric flow rate in the reactor (considered measurable and equal in both chambers); $c_{\text{cell}}(t)\!=\!n_{\text{cell}}(t)/(\epsilon V_{\text{cell}})$  and $\Delta SOC(t)\!=\! SOC_{\text{cell}}(t) \!-\!SOC(t)$ with $\epsilon$ accounting for the known porosity of the carbon electrode. Thus, from \eqref{soc}-\eqref{soc_cell}, the resulting space-state isothermal lumped parameter model for the DRFB is

\noindent
\begin{align}
&   \begin{bmatrix} \frac{d SOC}{dt}(t) \\[0.1cm] \frac{d SOC_{\text{cell}}}{dt}(t)  \end{bmatrix} = \underbrace{\begin{bmatrix} 0 & 0\\ \frac{Q(t)}{\epsilon V_{\text{cell}}} & -\frac{Q(t)}{\epsilon V_{\text{cell}}} \end{bmatrix}}_{A(Q(t))} \begin{bmatrix}  SOC(t) \\ SOC_{\text{cell}} (t)\end{bmatrix} + \nonumber \\
\label{model_x}
& \underbrace{\begin{bmatrix}  -\frac{1}{c_0 V_{\text{res}}} \\[0.1cm] -\frac{1}{\epsilon c_0 V_{\text{cell}}} \end{bmatrix}}_{E} Q_x(s(t)) + \underbrace{\begin{bmatrix}  -\frac{1}{c_0 V_{\text{res}} \mathcal{F}} \\[0.1cm] -\frac{1}{\epsilon c_0 V_{\text{cell}} \mathcal{F} } \end{bmatrix}}_{B} I(t), \\
\label{model_y}
& V_{\text{out}} (t) = \underbrace{E^{\text{0}}_{\text{cell}} \!+\! \frac{2\mathcal{RT}}{\mathcal{F}} \ln{\left(\frac{SOC_{\text{cell}}(t)}{1\!-\!SOC_{\text{cell}}(t)}\right)} \!+\! V_R(t)}_{\Gamma(SOC_{\text{cell}}(t),I(t))},
\end{align}

\noindent
where $V_{\text{out}}$ stands for the measurable output voltage of the battery in accordance with the Nernst equation, $\mathcal{R}$ is the universal gas constant, $\mathcal{T}\!=\!275 \ ^{\circ} \text{K}$ and $V_R=V_R(SOC_{\text{cell}}(t),I(t))$ accounts for voltage drops due to overpotentials, likely caused by kinetic and resistive phenomena ($V_R(SOC_{\text{cell}}(t),0)=0$).

\subsection{General Approximate Model}
Due to the nature of the electrochemical phenomena involved in RFBs, it can be considered that the unknown function for crossover flow $Q_x \! \in \! \mathcal{C}(\Omega; \mathbb{R})$, with $\Omega$ a compact set, and $s \! : \! \mathbb{R}^+ \!\! \to \! \Omega$ a vector of $k \! \in \! \mathbb{N}$ battery variables, which could be states, inputs or outputs of the model, or even other measurements, independent of the model. Thus, using the universal approximation property of radial basis functions (RBFs) \cite{Park-1991} or fuzzy inference systems \cite{Jang-1990} (in accordance with Stone-Weierstrass approximation theorem \cite{Atkinson-2009}),

\noindent
\begin{align}
\label{Q_app}
Q_x(s(t))= \Psi(s(t)) \theta +  \varepsilon(s(t)),
\end{align}

\noindent
$\forall \ t \! \in \! [t_i, t_{i+1}]_{i \in \mathbb{N}}$, with  $\theta \!\in\! \mathbb{R}^m$ being a piece-wise constant vector of parameters and $\varepsilon$ an arbitrarily small approximation error in accordance with $m \!\in\! \mathbb{N}$ number of bounded basis functions $ \psi_j \!:\! \overline{\Omega} \! \supseteq \! \Omega \! \to \! \mathbb{R}$ selected $\forall \ j \! \in \! [0,m]$, where $\Psi\=[\psi_1,\ldots, \psi_m]$.

\vspace{0.2cm}

Let $x=[SOC, SOC_{\text{cell}}]^{\top}$ be the vector of battery states. Using the above-described approximation of the crossover function, \eqref{model_x}-\eqref{model_y} can be formulated as

\noindent
\begin{align}
\dot{x} (t) =& A(Q(t)) x(t) + E \Psi(s(t)) \theta(t) +  B I(t) + E \varepsilon(s(t)), \nonumber\\
\label{mod_app} \dot{\theta} (t) =& 0^{\top},  \qquad \quad \forall \ t \in [t_i, t_{i+1}],\\[-0.5cm]
y (t) =& \Gamma^{-1}(V_{\text{out}}(t),I(t),V_R(t))=\overbrace{\begin{bmatrix} 0 & 1 \end{bmatrix}}^{C} x(t) + w(t), \nonumber
\end{align}

\noindent
with\footnote{In terms of the observer design, matrices $A,B,C$ and $E$ are considered known. The inversion of $\Gamma$ is based on Condition \ref{c_inv}, for small errors in the voltage measurements.}  $x(0)\!=\!x_0$ and $\theta(0)\!=\!\theta_0$ being unknown state and parameter initial conditions, respectively, $y$ being the measurable battery output, $w$ the measurement error, and $\bigcup_{i \in \mathbb{N}} [t_i, t_{i+1}] \!=\! \mathbb{R}^+$ for consecutive time intervals. This model, however, must satisfy the following conditions:

\begin{cond}
\label{c_inv}
The functional structure of the $V_R$ term in \eqref{model_y} is such that the non-linear mapping $\Gamma$ is globally invertible\footnote{In particular, based on the strictly increasing (monotonic) property of the logarithm function, this is satisfied under open-circuit conditions (zero applied current and internal self-discharge) since $V_R(SOC_{\text{cell}}(t),0)=0$, for $SOC_{\text{cell}}(t) \neq 1, \ \forall \ t \in \mathbb{R}^+$.}. In addition, due to errors in the measurements, the inversion of this mapping is such that $w$ in \eqref{mod_app} is always bounded: $\sup_{t}\{\|w(t) \| \} \leq \bar{w}$, $\forall \ t \in \mathbb{R}^+$, for some $\bar{w} \in \mathbb{R}^{+}$.
\end{cond}

\begin{cond}
\label{c_par}
Crossover through the separator represents a slow degradation process such that in \eqref{mod_app} there exist  $[t_i,t_{i+1}]$ finite time intervals $\forall \ i \! \in \! \mathbb{N}$, and bounded parameters: $\sup_t\{\|\theta(t) \| \} \leq \gamma_{\theta}$, which lead to a bounded approximation error: $\sup_{t}\{\|\varepsilon(s(t)) \| \} \leq\bar{\varepsilon}$ in \eqref{Q_app},
$\forall \ t \in \mathbb{R}^+$, for some $\gamma_{\theta} \in \mathbb{R}^{+}$ and $\bar{\varepsilon} \in \mathbb{R}^{+}$.
\end{cond}

\begin{cond}
\label{c_lip} The vector of basis functions $\Psi$ in \eqref{Q_app} is always bounded: $\sup_t\{\|\Psi(z(t))\|\} \leq \! \gamma_{\Psi}$, $\forall \ t \! \in \! \mathbb{R}^{+}$ and every vector function $z: \! \mathbb{R}^+ \! \to \! \mathbb{R}^k$, with $\gamma_{\Psi}=\sup_{t \in \mathbb{R}^+}\{\|\Psi(z(t))\|\} \in \! \mathbb{R}^{+}$. In addition, this satisfies the Lipchitz condition with respect to the battery states/variables, namely $\|\Psi(s(t))\!-\!\Psi(\hat{s}(t))\| \leq \gamma_{\tilde{\Psi}} \|s(t)-\hat{s}(t)\|$, for some $\gamma_{\tilde{\Psi}} \! \in \! \mathbb{R}^{+}$, $\forall \ s\!:\! \mathbb{R}^+ \to \Omega$ and $\hat{s} \!:\! \mathbb{R}^+ \to \overline{\Omega} \supseteq \Omega$,  $\forall \ t \in \mathbb{R}^+$.
\end{cond}
\begin{cond}
\label{c_Q}
The volumetric flow rate $Q$ is a bounded measurable variable with $Q_{m}\!=\!\inf_{t} \{Q(t)\}>0$ and $Q_{M}\!=\!\sup_{t} \{Q(t)\}$, so that its domain of operation $\mathcal{Q}\!=\![Q_{\text{m}}, Q_{\text{M}}]$ is a compact set known in advance. Thus, the time-varying matrix $A$ in \eqref{model_x} can always be embedded in a polytope of matrices $\mathcal{A}$ where

\noindent
\begin{align*}
A(Q(t)) \in \mathcal{A} &= \mathbf{Co} \{A(Q_{\text{m}}), A(Q_{\text{M}})\} \\
\end{align*}

\vspace{-0.6cm}

\noindent
$\forall \ t \in \mathbb{R}^+$, where $\mathbf{Co}$ denotes the convex hull (minimal convex polytope) \cite{Boyd-1994,Anstett-2009}.
\end{cond}
%

%---------------------------------------------------------------------%
%---------------------------------------------------------------------%
\section{Adaptive Observer Design}
\label{Sobs}

For the approximate model \eqref{mod_app}, the observer design can be carried out using Lyapunov stability theory \cite{Khalil-2002} applied to state and parameter estimation of nonlinear uncertain systems \cite{Walcott-1987, Marino-1990, Kim-1997,Cho-1997, Arcak-2001, Ascencio-2004, Millerioux-2004}. The aim is to achieve stability in the sense of uniformly ultimately bounded (UUB) convergence of the observer estimation error.

\subsection{Adaptive Observer}
An adaptive Luenberger-type observer \cite{Gildas-2007} for simultaneous estimation of battery state and parameters for the system \eqref{mod_app} has the structure

\noindent
\begin{align}
\dot{\hat{x}}(t) &= A(Q(t)) \hat{x}(t) + E \Psi(\hat{s}(t))
\hat{\theta}(t) + BI(t) +  L \tilde{y}(t), \nonumber \\[-0.2cm]
\label{mod_obs} \dot{\hat{\theta}} (t) &= \Phi (\tilde{y}(t),\Psi(\hat{s}(t))),\\
\hat{y}(t) &= C \hat{x}(t), \nonumber
\end{align}

\noindent
where $\hat{x}$, $\hat{\theta}$ and $\hat{y}$ are estimated
states, parameters and outputs, respectively, $\hat{x}(0)\!=\!\hat{x}_0$ and $\hat{\theta}(0)\!=\!\hat{\theta}_0$ are estimated initial conditions, $\hat{s}$ stands for the estimated states/variables of the battery chosen to model the crossover flux, $\tilde{y}(t)\!=\!y(t)\!-\!\hat{y}(t)$ denotes the output estimation error, $L \in \mathbb{R}^{n}$ is the observer linear feedback gain and $\Phi$ is the adaptive compensation function to adapt parameters in \eqref{Q_app} as the battery evolves.

\subsection{Observer Estimation Error}
The dynamical error between the model \eqref{mod_app} and the proposed observer \eqref{mod_obs} can be written as:

\noindent
\begin{align}
\label{err_p}
\begin{split}
\dot{\tilde{x}}(t) &= \bar{A}(t) \tilde{x} (t)+ E \eta(t) - L w(t), \\
\eta(t) &= \tilde{\Psi}(t) \theta (t) + \hat{\Psi}(t) \tilde{\theta}(t) + \varepsilon(t) \\
\dot{\tilde{\theta}} (t) &= -\Phi(\tilde{y}(t),\hat{\Psi}(t)), \qquad \qquad \qquad \forall \ t \in [t_i, t_{i+1}], \\
\tilde{y}(t)& =C\tilde{x}(t)+w(t),
\end{split}
\end{align}

\noindent
where $\tilde{x}(t)\!=\!x(t)\!-\!\hat{x}(t)$,
$\tilde{\theta}(t)\!=\!\theta(t)\!-\!\hat{\theta}(t)$ and $\tilde{y}(t)\!=\!y(t)\!-\!\hat{y}(t)$ are estimation errors for states, parameters and outputs, respectively, $\tilde{x}(0)\!=\!\tilde{x}_0$ and $\tilde{\theta}(0)\!=\!\tilde{\theta}_0$ are initial error conditions; $\tilde{\Psi}(t)\!\triangleq\!\Psi(s(t))\!-\!\Psi(\hat{s}(t))$, $\hat{\Psi}(t)\!\triangleq\!\Psi(\hat{s}(t))$ and $\varepsilon(t) \triangleq \varepsilon(s(t))$; $\bar{A}(t) \triangleq A(Q(t))-LC$ is the closed-loop compensated matrix.

\begin{proper}
\label{p_bound}
Given the error dynamic \eqref{err_p}, the expression

\noindent
\begin{align}
\label{delta}
\delta(t) &= P E  \left((1-\rho) \tilde{\Psi}(t) \theta(t) + \varepsilon(t)\right)  -Z w(t)
\end{align}

\noindent
represents part of the uncertain dynamic which cannot be compensated by the observer feedback gain nor adaptive function in \eqref{mod_obs}, with $P \in \mathbb{S}^n_+$, $Z \in \mathbb{R}^n$ and for some $\rho \in [0,1]$. If Conditions \ref{c_inv}-\ref{c_lip} are satisfied, the following upper bound hold:

\noindent
\begin{align*}
\| \delta(t) \| & \leq  \sup_{t \in \mathbb{R}^{+}} \{ \| \delta(t) \| \}, \\
& \leq \|PE\| \left( 2 (1-\rho) \gamma_{\Psi} \gamma_{\theta} + \bar{\varepsilon} \right) + \|Z\| \bar{w}= \bar{\delta}, \end{align*}

\noindent
$\forall \ t \in \mathbb{R}^+$, where the boundedness of the basis functions selected plays a relevant role.
\end{proper}

\subsection{Observer Feedback Gain and Adaptation Law}
\begin{lemma}
\label{lem_1}
Let $\upsilon(t) \triangleq \rho E \tilde{\Psi}(t) \theta(t)$ be a bounded signal for some $\rho \in [0,1]$, and state-variables $s$ in \eqref{Q_app} be such that $\|\tilde{s}\| \leq \gamma_{\tilde{s}} \|\tilde{x}\|$ for some $\gamma_{\tilde{s}} \in \mathbb{R}^+$. If Conditions \ref{c_par}-\ref{c_lip} are satisfied and $\alpha \in \mathbb{R}^{+}$ and $\beta \in \mathbb{R}^{+}$ are properly selected so that $\gamma^{2} \leq \alpha \beta$, then the following inequality holds:

\noindent
\begin{align*}
2 \tilde{x}(t)^{\top} P \upsilon(t) & \leq \tilde{x}(t)^{\top} \left(\alpha P P +
\beta I_n \right) \tilde{x}(t),
\end{align*}

\noindent
where $P \in \mathbb{S}^n_+$ and $\gamma \!=\! \rho \gamma_{E} \gamma_{\theta} \gamma_{\tilde{\Psi}} \gamma_{\tilde{s}}$ with $\gamma_{E}\!=\!\|E\|$, $\forall \ s\!:\! \mathbb{R}^+ \to \Omega$, $\forall \ \hat{s}\!:\! \mathbb{R}^+ \to \overline{\Omega} \supseteq \Omega$ and $ \forall \ t \in \mathbb{R}^+$.
\end{lemma}

\begin{proof}
Following similar arguments described in \cite{Reif-1999,Ascencio-2004}, based on the binomial property $(\sqrt{\alpha} \|\tilde{x}^{\top} \! P \|\!-\!\sqrt{\beta} \|\tilde{x}\|)^2 $ $=\! \alpha \|\tilde{x}^{\top} \! P \|^2 \!+\! \beta \|\tilde{x}\|^2 \!-\! 2 \sqrt{\alpha \beta} \|\tilde{x}^{\top} \!P \| \|\tilde{x}\| \geq \! 0$, using bounds stated in Conditions \ref{c_par}-\ref{c_lip} and considering $\|\tilde{s}\| \leq \! \gamma_{\tilde{s}} \|\tilde{x}\|$, the selection of $\alpha$ and $\beta$ such that $\gamma^{2} \leq \alpha \beta$ yields

\noindent
\begin{align*}
2 \tilde{x}^{\top} P v(t) & \leq 2 \|\tilde{x}^{\top} P\| \|v(t)\| \leq 2 \rho \|\tilde{x}^{\top} P\| \|E\| \|\tilde{\Psi}\| \|\theta\|, \\
& \leq 2 \rho \gamma_E \gamma_{\theta} \gamma_{\tilde{\Psi}}\|\tilde{x}^{\top} P\| \|\tilde{s}\| \leq 2 \sqrt{\alpha \beta} \|\tilde{x}^{\top} P \| \|\tilde{x}\|, \\
& \leq \alpha \|\tilde{x}^{\top} P \|^2 + \beta \|\tilde{x}\|^2 = \tilde{x}^T (\alpha P P + \beta I_n) \tilde{x}.
\end{align*}

\vspace{-0.5cm}
\end{proof}
\begin{theorem}
\label{teo} Let Conditions \ref{c_inv}-\ref{c_lip} be satisfied by \eqref{mod_app}. If some positive constants $\alpha$ and $\beta$ are chosen such that $\alpha \beta \geq \gamma^{2}$ and the solution for the coupled Linear Matrix Inequality and Equality problem:

\noindent
\begin{align}
&\begin{bmatrix} \-\!A(Q(t))^{\top} \! P \- P A(Q(t))\!+\!C^{\top}\!Z^{\top}\!+\! Z C \- \beta I_n \-W & \!\!\! \sqrt{\alpha} P \\ \sqrt{\alpha} P & \!\!\! I_n \end{bmatrix}  \!\succeq\! 0, \nonumber \\
\label{lmi-lme}
& \ P E \!-\! C^{\top} F= 0,
\end{align}

\noindent
exists for some $P \in \mathbb{S}^n_+$, $W \!\succeq\! \textrm{diag}(\omega)$, $\omega \!\in\! \mathbb{R}^{n}$ ($\omega_j\!>0\! \ , \forall j\!=\!1,\ldots,n,\omega_i\!>0\! \ , \forall i\!=\!1,\ldots,n$), $Z \in \mathbb{R}^{n}$ and $F \in \mathbb{R}$, $\forall \ Q \in \mathcal{Q}$, then the adaptive observer \eqref{mod_obs} using the feedback gain and adaptation function

\noindent
\begin{align}
\label{law_l} L &= P^{-1} Z , \\
\label{law_p} \Phi (\tilde{y}(t),\hat{\Psi}(t)) &= \Lambda^{-1} \left(\hat{\Psi}^{\top} F
\tilde{y}(t)-\frac{1}{2}\sigma \hat{\theta} \|\tilde{y}(t)\|  \right),
\end{align}

\noindent
respectively, for some $\sigma \in \mathbb{R}^{+}$ and $\Lambda \in \mathbb{S}^m_+$, guarantees that the state estimation error $\tilde{x}$ and the parameter estimation error $\tilde{\theta}$ have a UUB dynamic \cite[pp.~168,~346]{Khalil-2002}.

\end{theorem}

\begin{proof}
Let $V \!=\! \tilde{x}^{\top} \! P \tilde{x} + \tilde{\theta}^{\top} \! \Lambda \tilde{\theta}$ be a Lyapunov function candidate\footnote{For clarity, the time-dependence in most of the functions after this has been dropped.}. Its time derivative along the trajectory \eqref{err_p} yields

\noindent
\begin{align}
\dot{V} = \ & \tilde{x}^{\top} \!\left( \bar{A}(t)\!^{\top}
\!P \!+\! P \bar{A}(t) \right) \tilde{x} \!+\! 2 \tilde{x}^{\top} P E \eta(t) \ + \nonumber \\
\label{lyap_01}
& 2 \dot{\tilde{\theta}}^{\top} \Lambda \tilde{\theta} - 2 \tilde{x}^{\top} P L w(t).
\end{align}

\noindent
If the signal $\eta$ in \eqref{err_p} is decomposed by factor $\rho \in [0,1]$, then $\eta=\eta_0 + (1-\rho)\tilde{\Psi}\theta + \hat{\Psi} \tilde{\theta} + \varepsilon$, where $\eta_0=\rho \tilde{\Psi}\theta$ represents the uncertain term which can be compensated by the observer feedback term $L\tilde{y}(t)$ in \eqref{mod_obs}. Thus, considering \eqref{law_l} so that  $\bar{A}(t)=A(Q(t))\!-\!P^{-1}ZC$, \eqref{lyap_01} can be written as

\noindent
\begin{align}
 & \dot{V} = \tilde{x}^{\top} \left( A(Q(t))^{\top} P + P  A(Q(t)) \!-\! C^{\top}Z^{\top}  \!-\! ZC \right) \tilde{x} + \nonumber \\
\label{lyap_02}
& \underbrace{2 \tilde{x}^{\top} P E \hat{\Psi} \tilde{\theta} + 2 \dot{\tilde{\theta}}^{\top} \Lambda \tilde{\theta}}_{T_1} + \underbrace{2 \tilde{x}^{\top} P (\rho E \tilde{\Psi} \theta)}_{T_2} + 2 \tilde{x}^{\top} \delta(t).
\end{align}

\noindent
Based on the LME condition in \eqref{lmi-lme}, substituting \eqref{law_p} into \eqref{lyap_02} and considering bounds from Conditions \ref{c_inv}-\ref{c_lip},  the term $T_1$ can be upper bound by:

\noindent
\begin{align}
T_{1} =& \ 2 \tilde{x}^{\top} C^{\top} \! F \hat{\Psi} \tilde{\theta} \- 2\big(\Lambda^{-1} [\hat{\Psi}^{\top} F \tilde{y}\-(1/2)\sigma\hat{\theta}
\|\tilde{y}\| ] \big)^{\top} \! \Lambda \tilde{\theta}, \nonumber\\
=& \ \sigma \|C\tilde{x} + w\| (\theta-\tilde{\theta})^{\top}\tilde{\theta}\!-\! 2w^{\top}F(\Psi\!-\!\tilde{\Psi})\tilde{\theta}, \nonumber \\
\leq & \ \sigma \gamma_C \|\tilde{x}\| (\gamma_{\theta} \|\tilde{\theta}\|-\|\tilde{\theta}\|^2) + 2 \gamma_F\gamma_{\tilde{\Psi}} \|w\| \|\tilde{x}\|\|\tilde{\theta}\| + \nonumber \\
 \label{ada_01}  & \ \sigma \|w\|(\gamma_{\theta} \|\tilde{\theta}\|-\|\tilde{\theta}\|^2) + | -w^{\top} F\Psi \tilde{\theta}|, \\
\leq & \ \sigma \gamma_C \|\tilde{x}\|\Big(\overbrace{\Big[\gamma_{\theta} + \frac{2  \gamma_F\gamma_{\tilde{\Psi}}}{\sigma \gamma_C}\bar{w} \Big]}^{\gamma_1}\|\tilde{\theta}\|-\|\tilde{\theta}\|^2\Big) + \nonumber \\
& \ \sigma \bar{w} \Big(\underbrace{\Big[\gamma_{\theta}  + \frac{2 \gamma_F \gamma_{\Psi}}{ \sigma} \Big]}_{\gamma_2} \|\tilde{\theta}\|-\|\tilde{\theta}\|^2 \Big), \nonumber
\end{align}

\vspace{-0.2cm}

\noindent
where $\gamma_{C} \!= \!\|C\|$ and $\gamma_{F} \!=\! \|F\|$. With respect to the term $T_2$, by virtue of Lemma \ref{lem_1} for the positive constants $\alpha$, $\beta$ so that $\gamma^{2} \leq \alpha \beta$ and Property \ref{p_bound}, this leads to:

\noindent
\begin{align}
& \dot{V} \! \leq \! \tilde{x}^{\top} \!\Big( \underbrace{ \! A(Q(t))^{\top} \! P \+ P  A(Q(t)) \-C^{\top}\!Z^{\top}  \!\-ZC \+ \alpha P P \+ \beta I_n }_{T_3} \! \Big) \tilde{x} \nonumber \\[-0.1cm]
\label{lyap_03}
&  \+ \sigma \gamma_{C} \|\tilde{x}\| \left(\!\gamma_{1} \|\tilde{\theta}\|\-
\|\tilde{\theta}\|^{2}\right) \+ \sigma \bar{w} \left(\!\gamma_{2} \|\tilde{\theta}\| \-
\|\tilde{\theta}\|^{2}\right) \+ \ 2 \|\tilde{x}\|  \bar{\delta}.
\end{align}

\noindent
Imposing the condition $T_3(t) \!\leq\! -W$, for some $W \!\succeq\! \textrm{diag}(\omega)$, $\omega \!\in\! \mathbb{R}^{n}$ ($\omega_i\!>0\! \ , \forall i\!=\!1,\ldots,n$) and $\forall \ t \in \mathbb{R}^+$, a Riccati-like inequality is posed \cite{Reif-1999} which, by means of the Schur complement \cite{Boyd-1994,Cho-1997}, can be transformed to the constrained time-variant problem \eqref{lmi-lme}, convex in terms of the variables $\{ P, Z, F, W\}$. If this problem is feasible, \eqref{lyap_03} can be written as:

\noindent
\begin{align}
\dot{V} & \leq - \varrho \gamma_W \|\tilde{x}\|^{2} -
\sigma \gamma_{C} \|\tilde{x}\|\left(\|\tilde{\theta}\| \! - \! \frac{\gamma_1}{2}\right)^{2} - \sigma \bar{w} \|\tilde{\theta}\| \left(\|\tilde{\theta}\| - \gamma_2\right)\nonumber \\
\label{lyap_04}  & \ - \|\tilde{x}\|
\left( (1-\varrho) \gamma_W \|\tilde{x}\| - \frac{\sigma \gamma_{C}
\gamma_1^{2}}{4} - 2 \bar{\delta} \right),
\end{align}

\noindent
with $\gamma_W=\lambda_{\text{min}}(W)$, for some $\varrho \in [0,1]$, so that \eqref{lyap_04} is negative defined outside of the ball $B=B(0,[r_{\tilde{x}}, r_{\tilde{\theta}}])$ set forth by the radii:
\begin{align}
\label{radii}
\begin{split}
\|\tilde{x}\| &\geq
\frac{1}{4(1-\varrho)\gamma_W} \left(\sigma \gamma_C \gamma_{1}^{2}
+ 8 \bar{\delta}\right)=r_{\tilde{x}},
\\ \|\tilde{\theta}\| & \geq \max \left\{\gamma_2, \frac{\gamma_1}{2} +
\left( \frac{\gamma_{1}^{2}}{4} + \frac{2 \bar{\delta}}{\sigma \gamma_C}
 \right)^{1/2}  \right\} =r_{\tilde{\theta}}.
\end{split}
\end{align}
Therefore, according to the standard Lyapunov theorem:
$\|\tilde{x}\|$ and $\|\tilde{\theta} \| $ have a convergent dynamic which is UUB \cite{Kim-1997,Narendra-1987, Khalil-2002,Ioannou-1996}.
\end{proof}

\begin{remark}
Under the ideal situation of $\delta\!=\!0$ in (\eqref{delta}),  zero output measurement error ($w\!=\!0$) and $\sigma=0$, the equilibrium point $(\tilde{x},\tilde{\theta})\!=\!0^{\top}$ is uniformly stable and $\lim_{t \to \+\infty} \tilde{x}(t) \to 0$ can be proved on the basis of Barbalat's lemma. However, as is well-known, the convergence of the parameters to their actual values is only guaranteed under persistent excitation conditions, namely $c_0 I_n \geq \int_{t_0}^{t_0+\Delta t} \Psi(s(\tau)) \Psi(s(\tau))^{\top} d\tau \geq c_1 I_n$, for some $\Delta t \in \mathbb{R}^+$, $c_0 \in \mathbb{R}^+$ and $c_1 \in \mathbb{R}^+$ \cite{Ioannou-1996,Narendra-1989,Ascencio-2004}.
\end{remark}

\subsection{Polytopic Design Approach}
Considering that the volumetric flow rate in \eqref{model_x} satisfies Condition \ref{c_Q}, the approximate model \eqref{mod_app} admits a polytope description so that the time-variant nature of the LMI-LME problem \eqref{lmi-lme} can be recast as a polytopic convex problem \cite{Ascencio-2004,Anstett-2009,Millerioux-2004,Boyd-1994}:

\noindent
\begin{align}
\label{poly_lmi_lme}
& \min_{\bar{\alpha}, \gamma_{Z}, \gamma_{F}}
\left\{ \bar{\alpha} + \kappa_{Z} \gamma_{Z} + \kappa_{F}
\gamma_{F} \right\} , \\
& \text{subject to:} \nonumber \\[-0.1cm]
& \forall \ i \= 1,2 \! : \! \left\{ \! \! \!
\begin{array}{l}
\begin{bmatrix} \!-\! A^{\top}_{i} P \!-\! P A_{i} \!+\! C^{\top}Z^{\top} \!+\! Z C \!-\! \overline{W} & P \\ P & \bar{\alpha} I_n \end{bmatrix} \succeq 0,
\end{array} \right. \nonumber \\
& \qquad \qquad \ \begin{bmatrix} \gamma_{Z} I_n & Z \\ Z^{\top} & \gamma_{Z} \end{bmatrix} \succeq  0, \ \begin{bmatrix} \gamma_{F} & F \\ F & \gamma_{F} \end{bmatrix} \succeq 0, \nonumber \\
& \qquad \qquad \ \ P=P^{\top} \succeq 0, W \succeq \text{diag}(\omega_1, \ldots,\omega_n), \nonumber\\
& \qquad \qquad \ \ P E - C^{\top} F=  0,\nonumber \\
& \qquad \qquad \ \ \bar{\alpha} >0, \gamma_{Z} \geq 0, \gamma_{F} \geq 0, \nonumber
\end{align}

\noindent
with $A_1 = A(Q_m)$ and $A_2 = A(Q_M)$ vertices of the convex hull $\mathcal{A}$, $\omega_j\!>0\! \ , \forall j\!=\!1,\ldots,n$; $\bar{\alpha}=1/\alpha$ degree of freedom to maximise the uncertainty allowed by the LMI-LME \eqref{lmi-lme}, for a given $\beta$ and provided its feasible solution $\{P,Z,F,W\}$, with $\overline{W}=\beta I_n+ W$. The constants $\kappa_{Z}\! \geq \! 0$ and $\kappa_{F} \! \geq \! 0$ are selected to adjust norm bounds of matrices $Z$ and $F$, respectively.

\begin{remark}
In the crossover approximation \eqref{Q_app} and consequent Lyapunov analysis of the observer estimation error \eqref{err_p}, the basis functions are not necessarily dependent on the measurable output variable. If this is the case, for zero output measurement error ($w\!=\!0$), the LMI-LME problem \eqref{lmi-lme} can be solved via a polytopic formulation of the Strictly Positive Real (SPR) condition on the resulting error dynamic \cite{Ioannou-1996}.
\end{remark}

%---------------------------------------------------------------------%
%---------------------------------------------------------------------%
\section{RFB Experimental Methods}
\label{lab}

The self-discharge experiment was performed inside an argon-filled glovebox (PureLab, Inert Technologies, USA) with 0.1 M vanadium acetylacetonate (V(acac)\textsubscript{3}, 98\%, Strem, UK) as the active species dissolved in anhydrous acetonitrile (ACN, 99.8\%, Sigma, UK) dried over molecular sieves (3 \AA, Sigma-Aldrich, USA) with 0.2 M tetraethylammonium tetrafluoroborate (TEABF\textsubscript{4}, Sigma, 99\%, UK) used as a supporting salt.

A flow cell designed for nonaqueous electrolytes was used with 2.20 cm\textsuperscript{2} of active area and reservoir volumes of 18 mL each. The flowrate of 9 mL/min corresponds to a reactor residence time of 4.65 s when using electrodes with a porosity $\epsilon$ of 0.87. Porous carbon felt electrodes (Alfa-Aesar, UK) compressed 50\% to a final thickness of 3.17 mm were used in a flow-through configuration alongside impervious bipolar graphite plates (GraphiteStore, USA). A polypropylene porous separator (Celgard 4650, Celgard, USA) was used to separate the counter-current electrolyte flow inside the cell.

Peristaltic pumps (MasterFlex, Cole-Palmer, USA) circulated electrolyte through the half-cells using polytetrafluoroethylene (PTFE) tubing with perfluoroalkoxy alkane (PFA) compression fittings. Wetted materials in the system consisted entirely of PTFE, PFA, polypropylene, impervious graphite, and carbon felt. Figure \ref{fig_exp} shows the operating experimental setup previously detailed.

\begin{figure}
    \centering
    \vskip+0.2cm
    \includegraphics[width=0.5\columnwidth]{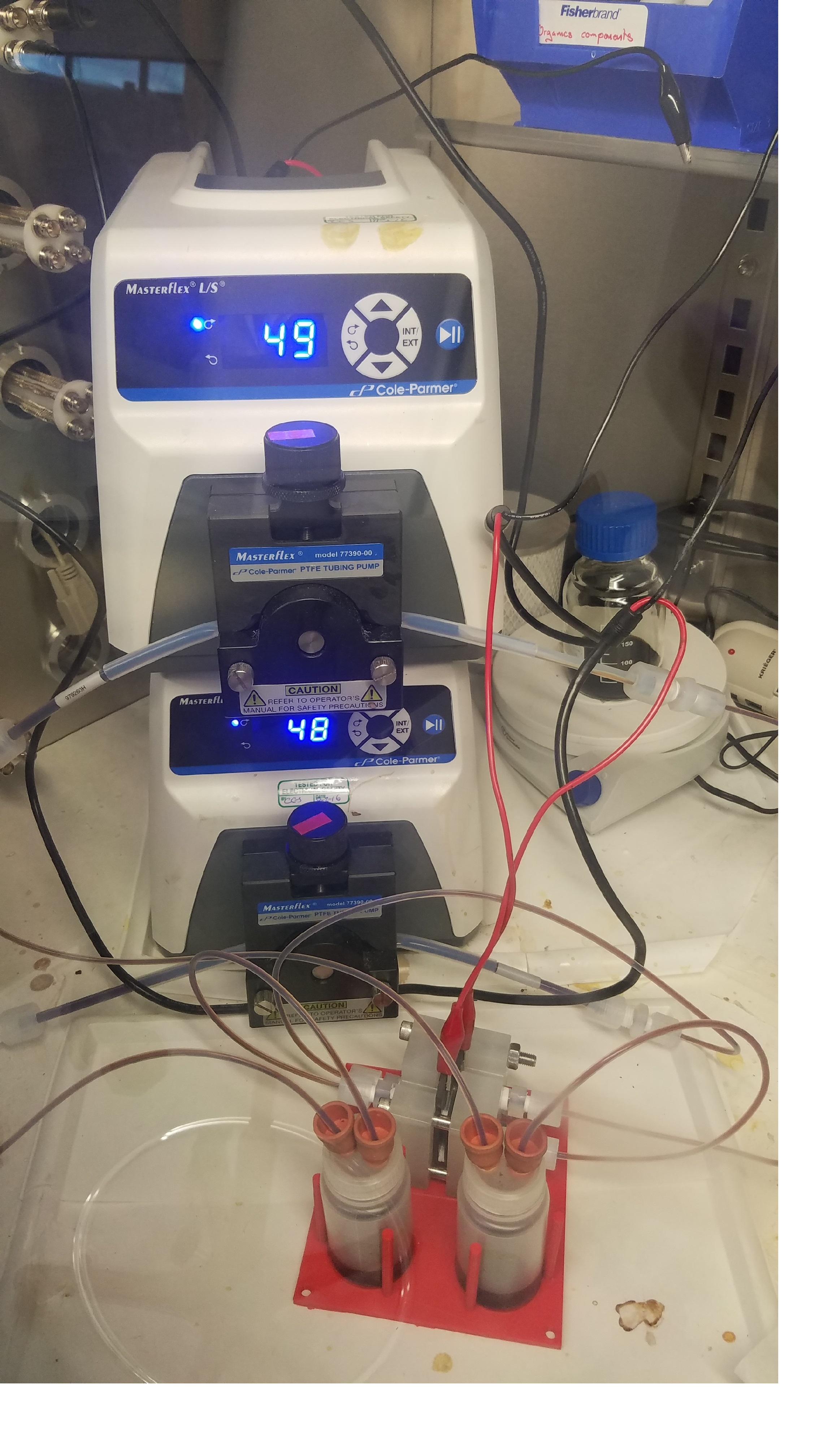}
    \vskip-0.1cm
     \caption{{\small RFB experiment running in glovebox}}
    \label{fig_exp}
    \vskip-0.1cm
\end{figure}

%---------------------------------------------------------------------%
%---------------------------------------------------------------------%
\section{Experimental Results}
\label{results}

The flow battery was cyled between 3 V and 1 V three times at 20 mA/cm\textsuperscript{2} with a battery tester (MACCOR 4000 series, USA) to precondition the system before charging to the 3 V voltage cutoff and letting the battery self-discharge at open circuit. Values of the experimental parameters in the model \eqref{model_x}-\eqref{model_y} are provided in Table \ref{tconst}.

For comparison and validation purposes, the following linear function for the crossover term in \eqref{model_x} has been considered:

\vspace{-0.2cm}

\noindent
\begin{align}
\label{lin_crossover}
Q_x(SOC_{\text{cell}}(t))=k_{\text{mt}} \ c_0 \overbrace{\left(\frac{c_0-c_{\text{cell}}(t)}{c_0}\right)}^{SOC_{\text{cell}}(t)},
\end{align}

\noindent
where $k_{\text{mt}}\!=\!5.6142 \! \cdot \! 10^{\!-8\!}$ L/min is the mass-transfer coefficient, a proposed constant which describes the expected linear relationship between $SOC_{\text{cell}}$ and $Q_x$. This value for $k_{mt}$ was calculated by fitting the model, \eqref{model_x}-\eqref{model_y}, including \eqref{lin_crossover}, to the experimental data assuming an initial condition of 100\% $SOC_{\textrm{cell}}$ at the beginning of the experiment and 50\% $SOC_{\textrm{cell}}$ when the output voltage of the battery reached the equilibrium voltage of 2.2 V. Figure \ref{fig_resul01} illustrates the model-predicted internal battery states and mass-transfer term, respectively (dash-dot lines). Figure \ref{fig_resul02} depicts the model voltage output as a function of the predicted $SOC_{\textrm{cell}}$, which is mostly in good agreement with its measured value (dotted line), with the exception of the extreme zones of $SOC_{\textrm{cell}}$, where unmodeled dynamics in the argument of the Nernst equation \eqref{model_y} are expected to be present.

Regarding the proposed on-line adaptive observer \eqref{mod_obs}, $\hat{x}_0\!=\![0.85,0.8]^{\top}$ and $\hat{\theta}_0\!=\!0^{\top}$ have been selected as its initial conditions. For the crossover approximation \eqref{Q_app}, $s\!=\!SOC_{\text{cell}}$, $\Omega \!=\![0,1]$ and seven ($m\!=\!7$) normalized radial basis functions, uniformly centred in $[0.05, 0.95]$ with variance $0.0081$, have been considered. With respect to the observer gains in \eqref{law_l}-\eqref{law_p}, $\Lambda^{\!-\!1}\!=\!4.798 \! \cdot \! 10^{\!-\!7} I_7$ and $\sigma\!=\!0.1$ have been selected to obtain to a slow parameter adaptation. The numerical solution of the polytopic LMI-LME problem \eqref{poly_lmi_lme}, considering $Q_{\text{m}}\!=\!0.9Q$, $Q_{\text{M}}\!=\!1.1 Q$, $\beta=10^{\!-\!4}$, $\kappa_F\!=\!10^{\!-\!5}$ and $\kappa_Z\!=\!1$, has been obtained via the Yalmip toolbox for Matlab \cite{Yalmip} using the SDP package part of the Mosek solver \cite{Mosek}. More details about its Matlab code implementation and data can be found in \cite{Matlab_code}. The simultaneous estimation of the state-of-charge and crossover flux is shown in Figure \ref{fig_resul01} (solid lines), achieving similar behaviour to the aforementioned model \eqref{model_x}-\eqref{model_y}-\eqref{lin_crossover}. In particular, the estimated crossover flux, whose parameter convergence to stationary values is shown in \ref{fig_resul01}(d), is in agreement with the linear relationship proposed in \eqref{lin_crossover}. Figure \ref{fig_resul02}, solid lines, illustrates the accurate on-line predicted output voltage based on the estimated states provided by the adaptive observer.

%-----------------------
\begin{table}[t]
\centering
\vskip+0.25cm
\begin{tabular}{ c | m{3.7cm} | l | l}
\hline
{\bf Symbol} & {\bf Description} & {\bf Value} & {\bf Units} \T\B \\
\hline \hline
$V_{\text{res}}$  			& Reservoir volume	 	                            & 17.6 & mL \T \\%$17.6 \cdot 10^{\!-\!3}$     & L \T \\
$V_{\text{cell}}$  		    & Half-cell volume 		                            & 0.6985 & mL \T \\%$9.3654 \cdot 10^{\!-\!4}$   & L \T \\
$c_0$  			            & Initial concentration  $\ce{V(acac)3}$            & $0.1$		                & mol/L \T \\
$Q$  			            & Volumetric flow rate 				                & 9.0 & mL/min \T \\%$1.5 \cdot 10^{\!-\!4}$      & L/s \T\\
%$k_{\text{mt}}$  		    & Mass-transfer coefficient                         & $3.3685 \cdot 10^{\!-\!6}$   &  L/s \T \\
$\epsilon$  			    & Porosity of carbon electrode                  & $0.87$		            &  -- \T \\
$E^0_{\text{cell}}$  	    & Equilibrium cell potential		                & $2.2$                     & V \T \B  \\
\hline
\hline
\end{tabular}
\vskip+0.1cm
\caption{{\small Parameters values for RFB experiment}}
\vskip-0.1cm
\label{tconst}
\end{table}
%-----------------------

\begin{figure}[thpb]
\vskip-0.5cm
\hskip-0.25cm
      \begin{tabular}{c}
       \includegraphics[width=9cm,height=6cm]{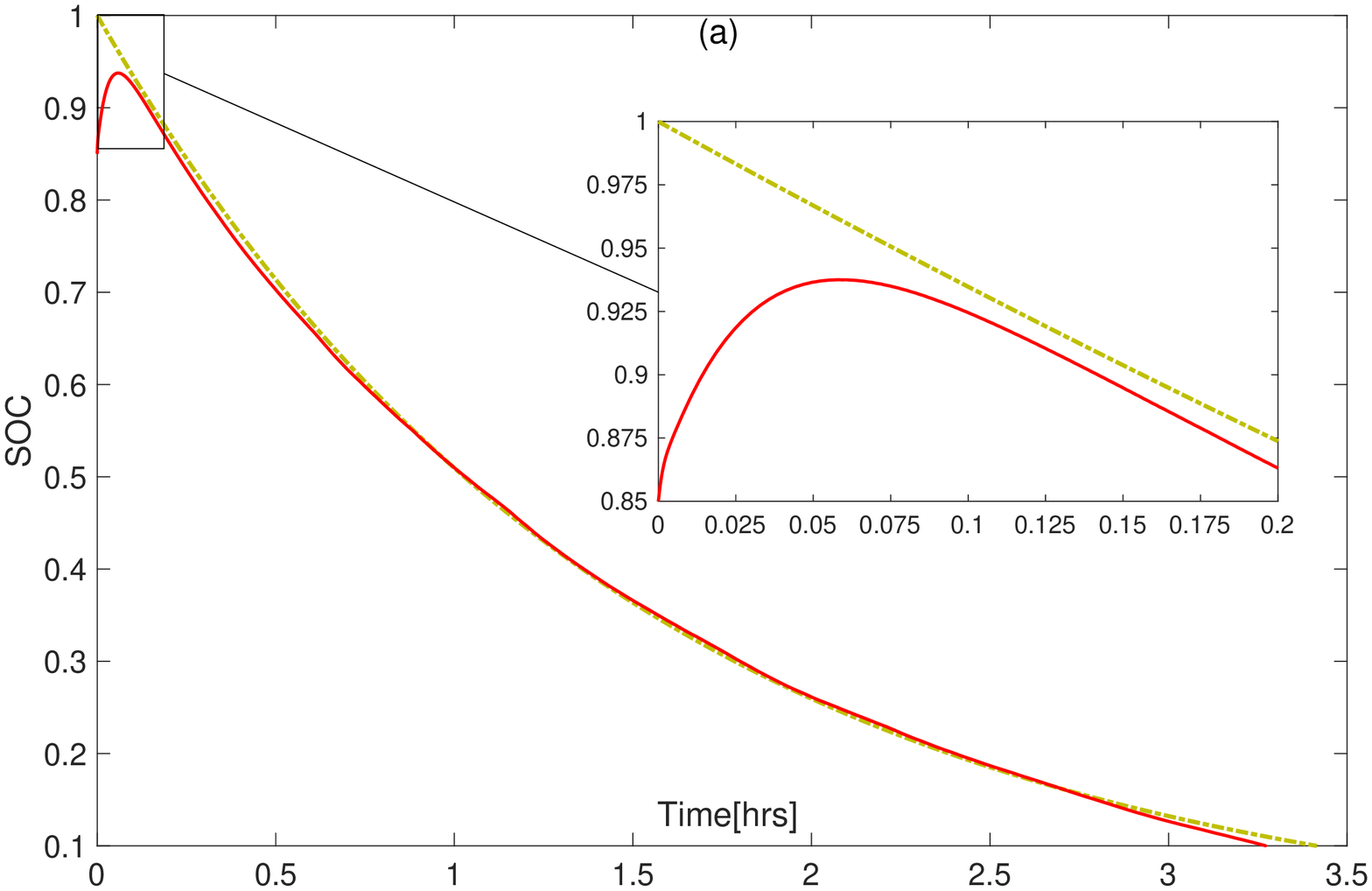} \\[-0.55cm]
       \includegraphics[width=9cm,height=6cm]{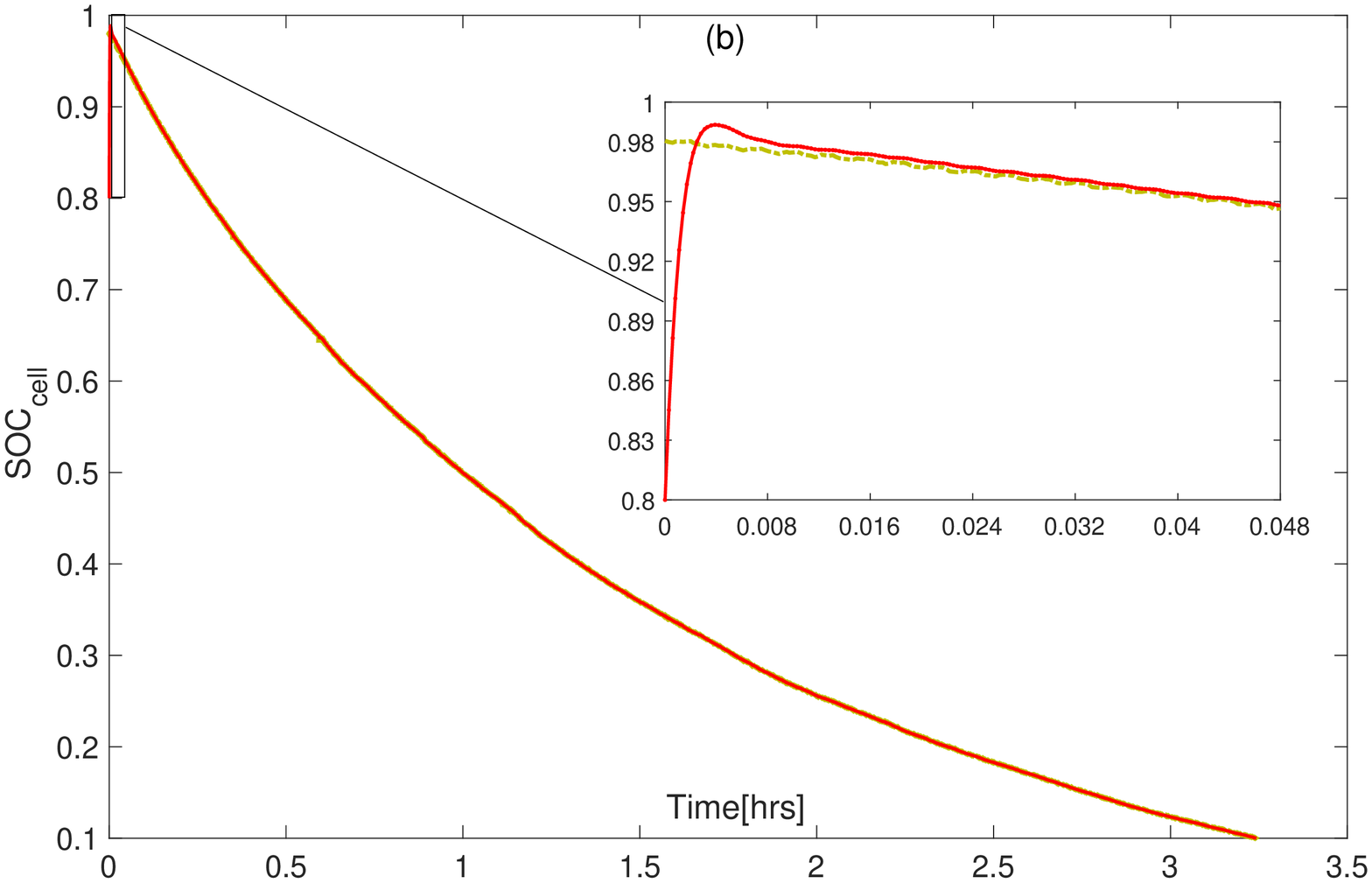}\\[-0.55cm]
       \includegraphics[width=9cm,height=6cm]{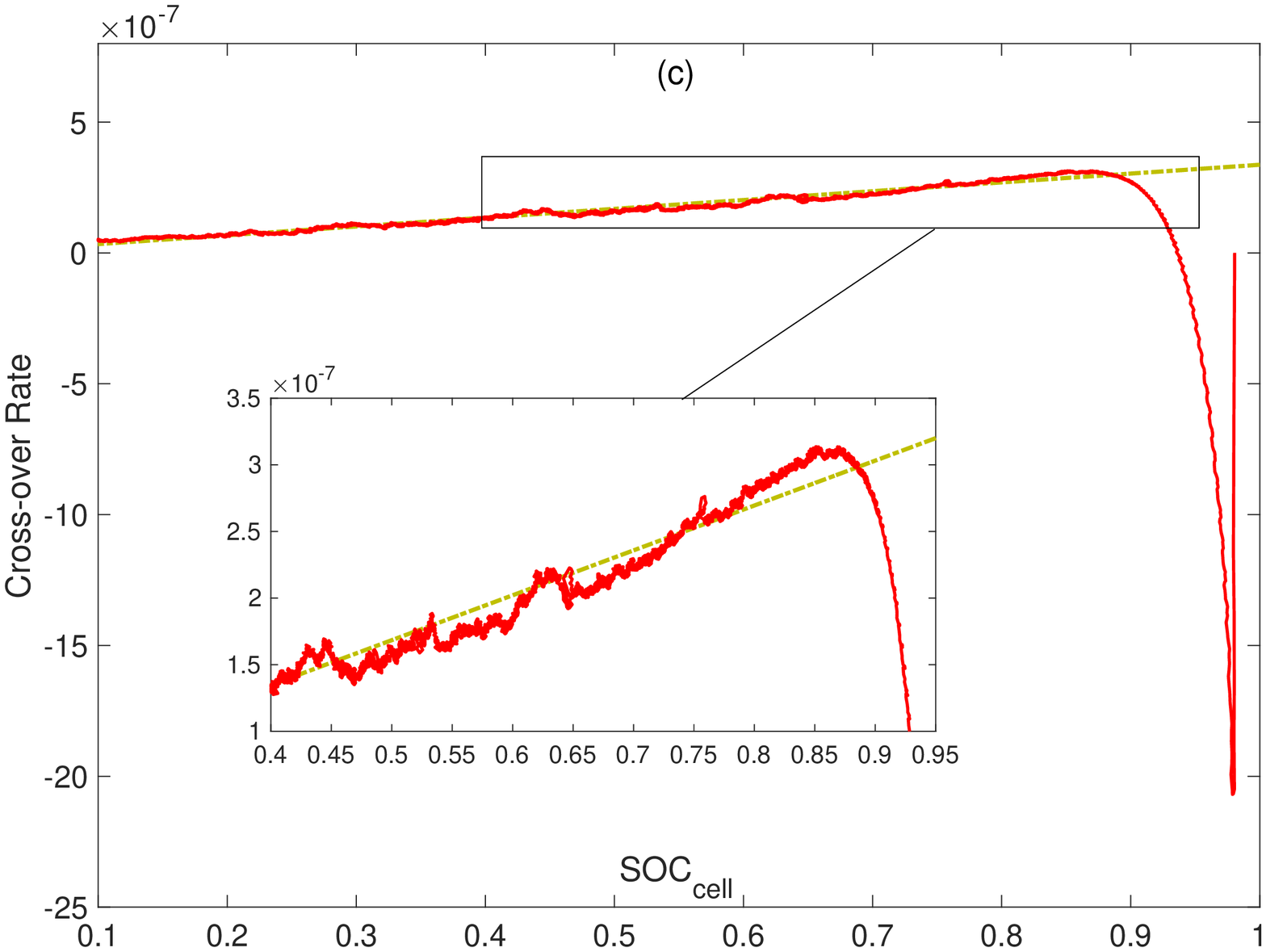}\\[-0.55cm]
       \includegraphics[width=9cm,height=6cm]{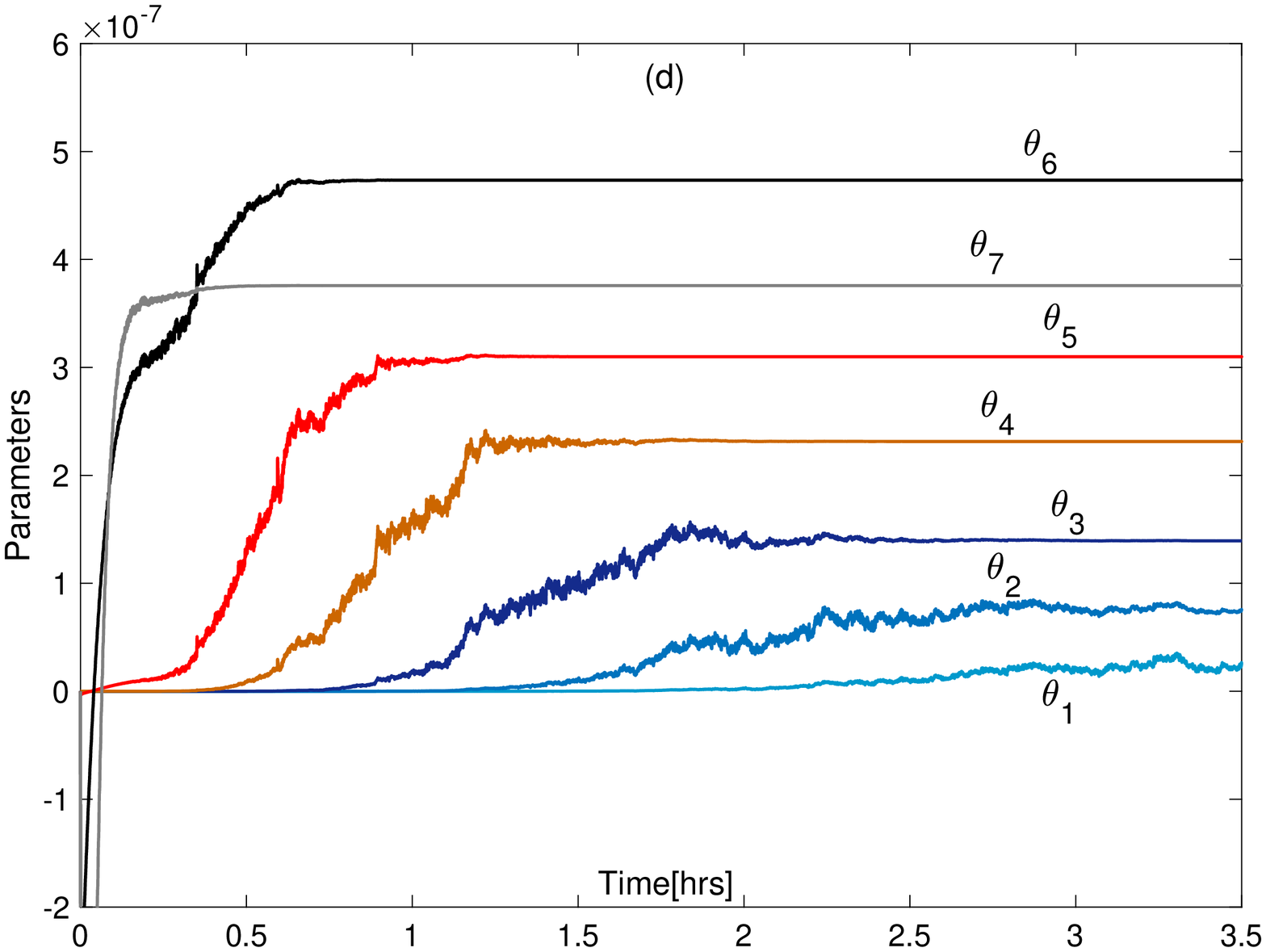}
       \end{tabular}
       \vspace{-0.25cm}
       \caption{\small{Battery model \eqref{model_x}-\eqref{model_y} results with constant parameters in  \eqref{lin_crossover} (dash-dotted lines), alongside adaptive observer estimate (solid lines) for: (a) overall state-of-charge; (b) state-of-charge of the half-cell; (c) crossover flux; (d) estimated parameters of approximation \eqref{Q_app}.}}
       \label{fig_resul01}
\vskip-0.5cm
\end{figure}

\begin{figure}[thpb]
\hskip-0.25cm
      \begin{tabular}{c}
       \includegraphics[width=9cm,height=6cm]{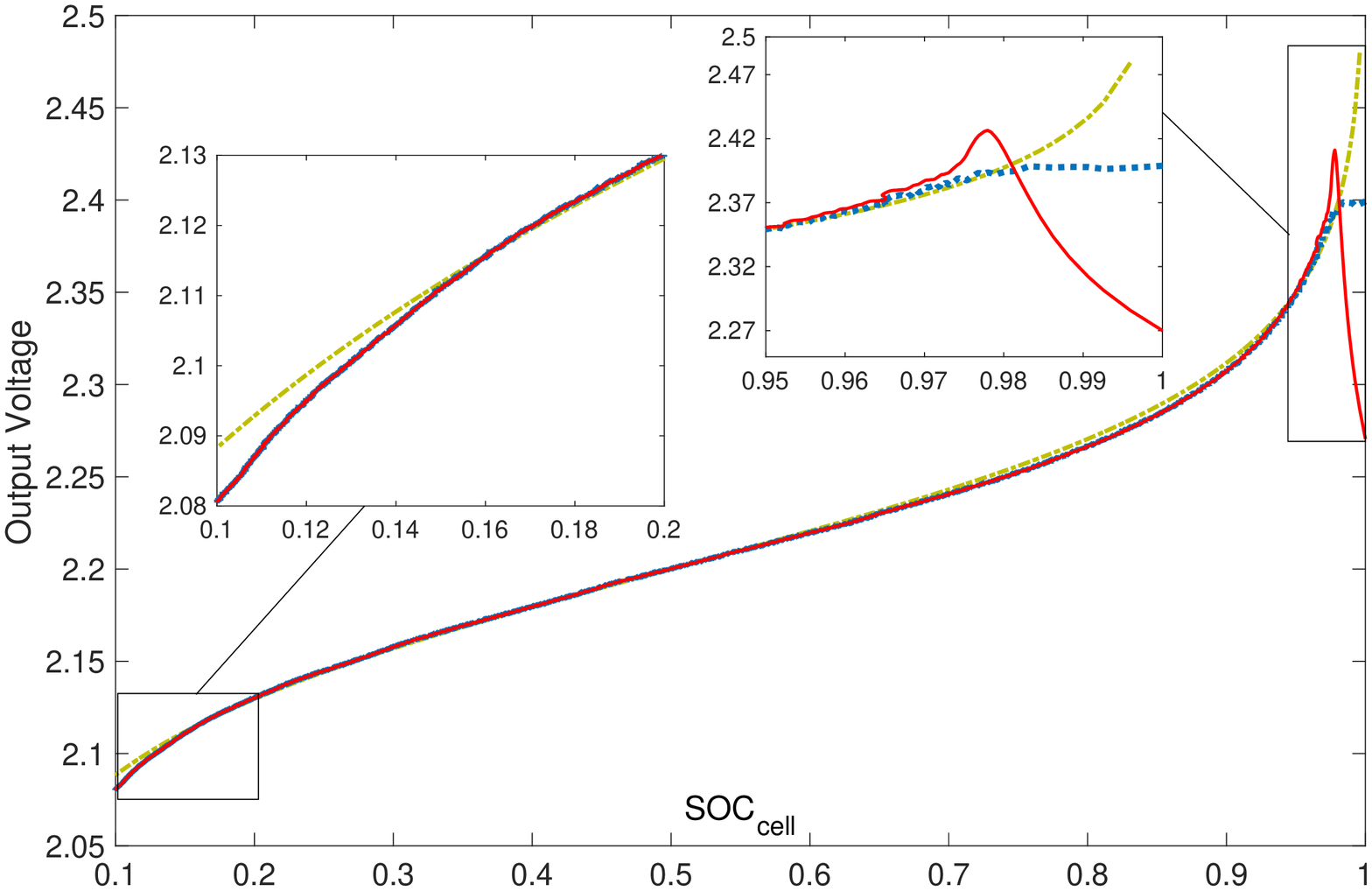}
       \end{tabular}
       \vspace{-0.3cm}
       \caption{\small {Voltages with respect to cell state-of-charge, including: measured voltage response (dotted line); battery model \eqref{model_x}-\eqref{model_y} with constant parameters in \eqref{lin_crossover} (dash-dotted line); and adaptive observer estimate (solid line).
       }}
       \label{fig_resul02}
\vskip-0.15cm
   \end{figure}

%%%%%%%%%%%%%%%%%%%%%%%%%%%%%%%%%%%%%%%%%%%%%%%%%%%%%%%%%%%%%%%%%%%%%%%%%%%%%%%%
\section{Concluding Remarks}
\label{conclu}
In this article an isothermal lumped model of a disproportionation redox flow battery and a related adaptive observer design have been presented. The design of the observer linear feedback gain and parameter adaptation law is based on Lyapunov stability theory and carried out by solving a polytopic LMI-LME problem. This provides a systematic methodology where the UUB convergence of state and parameter estimation errors is guaranteed. The crossover term has been approximated via radial basis functions, enabling continual estimation of the crossover flux as the battery discharges, which can then can be used predictively. A linear relationship between crossover flux and $SOC_{\textrm{cell}}$, captured by mass-transfer coefficient $k_{\textrm{mt}}$, is typically built into simple lumped-parameter RFB models, which assume pseudo-steady diffusion across a planar separator. In the adaptive observer, this relationship is not assumed, yet the observer reveals this linear dependence, as seen in Figure \ref{fig_resul01}. Data from a vanadium acetylacetonate DRFB undergoing self-discharge was analyzed, demonstrating the observer's capability to estimate state-of-charge and crossover simultaneously.

%%%%%%%%%%%%%%%%%%%%%%%%%%%%%%%%%%%%%%%%%%%%%%%%%%%%%%%%%%%%%%%%%%%%%%%%%%%%%%%%

\section{Acknowledgments}
This work was carried out with funding support received from the Faraday Institution ({\tt faraday.ac.uk}; EP/S003053/1), grant number FIRG003.

%%%%%%%%%%%%%%%%%%%%%%%%%%%%%%%%%%%%%%%%%%%%%%%%%%%%%%%%%%%%%%%%%%%%%%%%%%%%%%%%

%% Bibliography
%--------------------------
\bibliographystyle{unsrt}
\bibliography{ASMH_ACC2019_bib}

\end{document}